\documentclass{./irmaart}

\usepackage{graphicx}
\usepackage{pinlabel}

\numberwithin{equation}{section}

\theoremstyle{plain}

\newtheorem{theorem}{Theorem}[section]
\newtheorem{corollary}[theorem]{Corollary}
\newtheorem{lemma}[theorem]{Lemma}
\newtheorem{proposition}[theorem]{Proposition}

\theoremstyle{definition}
\newtheorem{definition}[theorem]{Definition}
\newtheorem{remark}[theorem]{Remark}

\newtheorem{example}[theorem]{Example}

\newcommand{\Ker}{\operatorname{Ker}}

\newcommand{\cC}{\mathcal{C}}
\newcommand{\cD}{\mathcal{D}}
\newcommand{\cF}{\mathcal{F}}

\newcommand{\cH}{\mathcal{H}}

\newcommand{\cM}{\mathcal{M}}
\newcommand{\cMbar}{\overline{\mathcal{M}}}
\newcommand{\cN}{\mathcal{N}}

\newcommand{\C}{{\mathbb{C}}}

\newcommand{\Q}{{\mathbb{Q}}}
\newcommand{\R}{{\mathbb{R}}}

\newcommand{\Z}{{\mathbb{Z}}}
\newcommand{\p}{{\partial}}

\newcommand{\one}
{{{\mathchoice \mathrm{ 1\mskip-4mu l} \mathrm{ 1\mskip-4mu l}
\mathrm{ 1\mskip-4.5mu l} \mathrm{ 1\mskip-5mu l}}}}
\newcommand{\id}{\operatorname{id}}         

\newcommand{\Aut}{\operatorname{Aut}}          
\newcommand{\eps}{{\varepsilon}}
\newcommand{\del}{\partial}

\begin{document}

\address{Fachbereich Mathematik, Universit\"at Hamburg, Bundestrasse 55, 20146 Hamburg, Germany}

\title{Fukaya's work on Lagrangian embeddings}

\author{Janko Latschev}

\maketitle

\begin{abstract}
In this chapter I discuss some applications of string topology to the
study of Lagrangian embeddings into symplectic manifolds, as discovered by
Fukaya~\cite{Fu:06}.\\
\end{abstract}





\section{Introduction}

A submanifold $L \subset (X,\omega)$ in some symplectic
manifold $(X,\omega)$ is called Lagrangian if $\dim L = \frac 1 2 \dim
X$ and $\omega|_L=0$. A simple example is given by the zero section $L
\subset T^*L$ in the cotangent bundle of a smooth manifold $L$, and this is 
universal in the sense that a neighborhood 
of any Lagrangian embedding of a closed $L$ into some symplectic manifold is
symplectomorphic to a neighborhood of $L \subset T^*L$. Lagrangian
submanifolds play a fundamental role in symplectic geometry and topology,
as many constructions and objects can be recast in this form. In fact,
already in a 1980 lecture (cf.~\cite{We:81}), A. Weinstein formulated the
``symplectic creed'': 
\newline
\newline
\centerline{\em EVERYTHING IS A LAGRANGIAN
  SUBMANIFOLD.}\newline
\newline
Today, Lagrangian submanifolds (sometimes decorated with additional
structures) are for example studied as objects of the {\em Fukaya
  category}, which plays a fundamental role in Kontsevich's
formulation of homological mirror symmetry. 
Rather than delving into such general theories, I want to concentrate
here on a quite simple, and in fact basic, question:

\medskip

{\em Which closed, oriented  $n$-manifolds admit a
  Lagrangian embedding into the standard symplectic space $(\C^n,
  \omega_0)$, with $\omega_0 = \sum_j dx_j \wedge dy_j$?}

\medskip

An excellent introduction to this question, containing a discussion of some
of the relevant classical algebraic topology, as well as early results
obtained by holomorphic curve methods, is \cite{ALP:94}, which I will
quote freely.

For $n=1$ there is not much to say, since $S^1$ is the only connected
closed 1-manifold, and the Lagrangian condition $\omega_0|_L=0$ is
trivial in this case. In general, a necessary condition for an
oriented closed manifold $L^n$ to admit a Lagrangian embedding into
$\C^n$ is that its Euler characteristic $\chi(L)$ should vanish. This
is because the self-intersection number of any submanifold of $\C^n$
is clearly zero, but it is also equal to the Euler characteristic of
the normal bundle, which for Lagrangian submanifolds is isomorphic to
the cotangent bundle.  

So for $n=2$, the only orientable closed manifold that could have a
Lagrangian embedding into $\C^2$ is $T^2=S^1 \times S^1$, and it
embeds e.g. as the product of one circle in each $\C$-factor. For
non-orientable closed surfaces $\Sigma$, classical algebraic topology
implies that a necessary condition for the existence of a Lagrangian
embedding is that $\chi(\Sigma)$ is divisible by 4, and a beautiful
construction by Givental~\cite{Gi} shows that for strictly negative
Euler characteristic this is also sufficient. The embedding question
was only recently completely answered, when Shevchishin showed that
the Klein bottle does not have a Lagrangian embedding into $\C^2$ (\cite{Sh:09}, see also \cite{Ne:09} for an alternative argument by Nemirovski).

Already for $n=3$, elementary algebraic topology does not tell us
much. It was one of the many important results in Gromov's landmark
paper~\cite{Gro:85} to show that there are no exact Lagrangian
embeddings into $(\C^n,\omega_0)$, in the sense that any global
primitive $\lambda$ of the symplectic form $\omega_0$ has to restrict
to a non-exact closed 1-form on the Lagrangian submanifold $L \subset
\C^n$. This in particular rules out $S^3$, but of course there are
plenty of closed orientable 3-manifolds with $H^1(M,\R) \neq 0$.

All of this and more is discussed in \cite{ALP:94}. 
The goal of this chapter is to show how knowledge about string
topology can be applied to give a far-reaching refinement of
Gromov's result. In particular, I aim to present the overall
strategy for proving the following result:

\begin{theorem}(Fukaya)\label{thm:fukayamain}
Let $L$ be a compact, orientable, aspherical spin manifold of
dimension $n$ which admits an embedding as a Lagrangian submanifold of
$\C^n$. Then a finite covering space $\tilde L$ of $L$ is homotopy
equivalent to a product $S^1 \times L'$ for some closed $n-1$-manifold $L'$. 

Moreover, $\pi_1(\tilde L) \cong \pi_1(S^1 \times L') \subset
\pi_1(L)$ is the centralizer of some element $\gamma\in \pi_1(L)$
which has Maslov class equal to 2 and positive symplectic area. 
\end{theorem}

The assertion about the Maslov class is known as {\em Audin's conjecture}, and was originally asked for tori in $\C^n$, see \cite{Au:88}. 
The spin condition is a technical assumption (it is needed to make the
relevant moduli spaces of holomorphic disks orientable), and I expect
that it can be removed by reformulating the argument somewhat. The
asphericity assumption (meaning that all higher homotopy groups of $L$
vanish) enters the proof in a fairly transparent way, and one
can imagine various replacements. 

As a corollary, we obtain the following more precise statement in dimension 3.

\begin{corollary}(Fukaya)\label{cor:fukaya1}
If the closed, orientable, prime 3-manifold $L$ admits a
Lagrangian embedding into $\C^3$, then $L$ is diffeomorphic to
a product $S^1 \times \Sigma$ of the circle with a closed, orientable
surface. 
\end{corollary}
The fact that the product $S^1 \times \Sigma$ does embed as a
Lagrangian submanifold into $\C^3$ follows from an elementary
construction, see e.g.~\cite{ALP:94}. Basically, one starts from an
isotropic embedding of $\Sigma$ into $\C^3$, e.g. by embedding it into 
the Lagrangian subspace $\R^3 \subset \C^3$. Then one uses the fact 
that a small neighborhood necessarily is symplectomorphic to a
neighborhood of the zero section in $T^*\Sigma \oplus \underline{\C}$,
the direct sum of the cotangent bundle with a trivial symplectic vector
bundle of rank 2, to embed the product $S^1 \times \Sigma$ by taking
the product of the zero section in $T^*\Sigma$ with a standard small
$S^1 \subset \C$. 

The above statements are special cases of a more general result
discovered by Kenji Fukaya, and first described in \cite{Fu:06}, see
also \cite{Fu:07}.  
As with most results involving $J$-holomorphic curves, the underlying
idea can be traced back to Misha Gromov's foundational paper
\cite{Gro:85}. His proof of the fact that there are no exact 
compact Lagrangian submanifolds of $\C^n$ contains an important seed
for Fukaya's arguments. Therefore, after discussing some aspects of
moduli spaces of holomorphic disks in the next section, I begin
section~\ref{sec:main} by sketching Gromov's argument, followed by the
discussion of an instructive example and the statement of Fukaya's
refinement (Theorem~\ref{thm:blackbox}). In a nutshell, Fukaya's
important observation was that the compactification of the moduli
spaces of holomorphic disks with boundary on a Lagrangian submanifold
$L\subset \C^n$ can be expressed in terms of string topology
operations, in particular the loop bracket (and possibly also its
higher analogues at the chain level). 

After an interlude section on basic properties of L$_\infty$ algebras, 
I describe how Theorem~\ref{thm:fukayamain} is a fairly
straightforward consequence of Theorem~\ref{thm:blackbox}.
Corollary~\ref{cor:fukaya1} will follow by using specific
facts from 3-dimensional topology. Before finishing with a guide
to  the literature, I discuss a few further small observations.

At the moment of this writing, full proofs of the above theorems are
not available yet. One would like to construct a chain model
$C_*(\Lambda L)$ for the free loop space of the manifold $L$ with the
loop bracket operation $\lambda_2$ (and higher operations as
necessary), such that $(C_*(\Lambda L),\lambda_1=\del,
\lambda_2,\dots)$ is an L$_\infty$-algebra. Moreover when this model 
is considered with coefficients in a suitable Novikov ring, the
compactified moduli spaces of holomorphic disks should give rise to an
element in this chain complex satisfying the Maurer-Cartan equation in
this L$_\infty$-algebra. So the principal problem is that the apparent
freedom one has in building the chain model is severely constrained by
the need to make it fit with the analysis of holomorphic disks, which
in particular involves delicate transversality and gluing issues. 

In this text, I will largely ignore the technical difficulties,
by stating the key results as black boxes. 

\section{Moduli spaces of holomorphic disks}
\label{sec:disks}

\subsection*{Basic notions}

The main tool in our study of Lagrangian embeddings $L \subset \C^n$
will be moduli spaces of holomorphic disks with boundary on
$L$. Before introducing the relevant spaces, I will briefly review
some standard notions.

The first of these is the {\em Maslov index} of a loop of Lagrangian
subspaces in $\C^n$. The Lagrangian Grassmannian
$\operatorname{GLag}_n$ is defined as the space of all Lagrangian
subspaces of $(\C^n, \omega_0)$. Standard symplectic linear algebra
shows that a real $n$-dimensional subspace $V \subset \C^n$ is Lagrangian
if and only if it is orthogonal with respect to the standard Euclidean
inner product to $iV$, its rotation by $i\in \C$. Moreover, any orthonormal
basis of a Lagrangian subspace is a unitary basis of $\C^n$, and
conversely the real linear span of a unitary frame is a Lagrangian
subspace. Since $O(n)$ acts transitively on the unitary frames generating the
same Lagrangian subspace, one concludes that $\operatorname{GLag}_n$
can be identified with $U(n)/O(n)$. 

The map $\det^2:U(n) \to S^1$ which associates to a unitary
matrix the square of its determinant induces a well-defined map 
$\rho:\operatorname{GLag}_n \to S^1$, and for any loop $\gamma:S^1 \to
\operatorname{GLag}_n$ we define {\em its Maslov index} as
$$
\mu(\gamma) := \deg(\rho \circ \gamma).
$$
Clearly $\mu(\gamma)$ so defined is an invariant of the free homotopy
class of $\gamma$. Moreover, one can check that $\mu$ induces an
isomorphism  $\pi_1(\operatorname{GLag}_n) \cong \Z$.

Now given a Lagrangian immersion $L \hookrightarrow \C^n$, any loop
$\gamma:S^1 \to L$ determines a loop in $\operatorname{GLag}_n$, simply
by taking the loop of tangent spaces of $L$ at the image points of
$\gamma$. In this way, we get a Maslov index
$$
\mu:\pi_1(L) \to \Z.
$$
Discussions of the Maslov index can be found in various texts,
see e.g. \cite[p.~116ff]{AG:85}, \cite[p.~33ff]{AL:94} or
\cite[p.~50ff]{MS:98}. In particular, it is easy to see that for a
Lagrangian immersion of an oriented manifold $L$, the Maslov index
takes values in $2\Z$. 

\bigskip

Another useful notion is the {\em area} or {\em energy} of a disk $u:(D,\p D)
\to (\C^n,L)$ with boundary on a Lagrangian submanifold $L \subset
(\C^n,\omega_0)$. It is defined as
$$
E(u) := \int_D u^*\omega,
$$
and one easily checks that it only depends on the free relative homotopy class
of $u$. Indeed, just observe that given a homotopy $h:([0,1] \times
D, [0,1] \times \p D) \to (\C^n,L)$ with $h_0=u$ and $h_1=u'$, we have
$$
\int_D u'^*\omega - \int_D u^*\omega = \int_{[0,1] \times \p D} h^*\omega
= 0
$$
by the assumption that $L$ is a Lagrangian submanifold. In particular,
$E$ descends to a homomorphism $E: \pi_2(\C^n, L) \to \R$.

\bigskip

Finally, in order to discuss holomorphic curves, we introduce the relevant 
spaces of almost complex structures. Generally, an {\em almost complex structure} $J$ on some manifold $M$ is an automorphism of the tangent bundle $J:TM \to TM$ with $J^2=-\id$. If $M$ carries a symplectic form $\omega$, then an almost complex structure $J$ on $M$ is said to be {\em tamed by} $\omega$, if $\omega(v,Jv)>0$ for all nonzero tangent vectors $v$. In other words, $\omega$ is a positive area form on each 1-dimensional $J$-complex subspace of each tangent space. It is a standard result that these $\omega$-tamed almost complex structures form a contractible (in particular non-empty) space. Given such a tamed $J$, one can define a Riemannian metric on $M$ by setting
$$
g_J(v,w):= \frac 12\left(\omega(v,Jw) + \omega(w,Jv)\right).
$$
One can also define the notion of a $J$-holomorphic map $u:(\Sigma,j) \to (M,J)$ from some Riemann surface $(\Sigma,j)$ to $M$, simply by asking that it satisfy the usual Cauchy-Riemann equations with respect to the given $J$:
\begin{equation}
\bar\p_Ju:= \frac 12(du+J\circ du\circ j)=0.
\end{equation}
In local conformal coordinates $z=s+it$ on the Riemann surface this can be written equivalently as
$$
\p_su + J \p_tu = 0.
$$
In Riemannian geometry, the $L^2$-energy of a map $u:\Sigma \to (M,g_J)$ is defined as 
$$ 
E(u) = \frac 12 \int_\Sigma \|du\|^2 d\mu,
$$
where $\mu$ is any volume form on $\Sigma$ and $\|du\|$ denotes the operator norm of $du:T_x \Sigma \to T_{u(x)}M$ with respect to the metric $\mu(\,.\,,j\,.\,)$ on $\Sigma$ and the metric $g_J$ on $M$. The integrand turns out to be independent of the choice of $\mu$, as any scaling factor also appears in the operator norm with opposite exponent. 

Now the importance of the taming condition stems from the following crucial fact: Suppose $u:\Sigma \to M$ is $J$-holomorphic, and we have chosen local conformal coordinates $(s,t)$ on $\Sigma$.  Then at $x$ we have
$$
\|du\|^2 = \omega(\p_su,J\p_su)+ \omega(\p_tu,J\p_tu)= 2 \omega(\p_su,\p_tu),
$$
so that 
$$
\frac 12\|du\|^2ds \wedge dt = u^*\omega.
$$
In particular, for $J$-holomorphic disks $u:(D,\p D) \to (\C^n,L)$, the energy as defined above, which was a purely topological quantity, is the same as the usual $L^2$-energy of the map $u$.

\subsection*{Moduli spaces}


Choose an almost complex structure $J$ tamed by the standard
symplectic structure $\omega_0$ on $\C^n$. Given a relative homotopy
class $a\in \pi_2(\C^n,L)$, we consider the set
$$
\widetilde{\cM}(a,J):= \{ u:(D,\p D) \to (\C^n, L) : \bar\p_J u =
0, [u] = a \in \pi_2(\C^n,L)\}.
$$
The real 2-dimensional group $\Aut(D,1) \subset PSL(2,\R)$ of biholomorphisms of the disk fixing $1\in S^1$ acts on $\widetilde{\cM}(a,J)$ by
precomposition, and the quotient 
$$
\cM(a,J):= \widetilde{\cM}(a,J)/\Aut(D,1)
$$ 
is called the moduli space of holomorphic disks in the class $a$. 
\begin{remark}
The reader should take note that in the literature moduli spaces of
holomorphic curves are almost universally denoted by $\cM$, usually
with some decoration, and the precise meaning of the symbol should
very carefully be checked in each case.
\end{remark}
The equation defining $\widetilde{\cM}(a,J)$ is elliptic, and so the
linearization is a Fredholm operator. The index theorem for
holomorphic curves, as discussed for example in \cite[Appendix
C]{MS:04}, implies  that the index of this operator, and hence the
expected dimension of the moduli space $\widetilde{\cM}(a,J)$, is $n+
\mu(a)$. The assignment $u \mapsto \bar\p_Ju$ can be viewed as a section 
of a suitable Banach space bundle. Under favourable circumstances 
this section can be arranged to be transverse to the zero section, in which 
case  its zero set $\widetilde{\cM}(A,J)$ is a manifold of the expected 
dimension. In 
general, this is a serious technical difficulty beyond the scope of the 
present discussion, and resolving it requires substantial work.
 
Note that, for $a\neq 0$, the action of $\Aut(D,1)$ on
$\widetilde{\cM}(a,J)$ is free. So, assuming that we can arrange
transversality, we conclude that $\cM(a,J)$ is a smooth manifold of
dimension 
\begin{equation}\label{eq:index}
\dim \cM(a,J) = n-2 + \mu(a).
\end{equation}
As shown by examples in \cite{FOOO}, one generally needs the spin
condition on $L$ to be able to orient the moduli spaces $\cM(a,J)$. 

For several reasons the individual spaces $\cM(a,J)$ tend to
not be very useful for proving anything interesting about $L$ (with the
notable exception of Theorem~\ref{thm:gromov-notexact}). The first
is that these spaces strongly depend on $J$, as simple examples show.
\begin{example}\label{ex:nosol}
Consider $L=S^1 \times S^1 \subset \C^2$ and the relative homotopy
class $a=(1,0) \in \pi_2(\C^2,L)$ which has degree 1 in the first
factor and degree 0 in the second. 

Given any real number $\alpha \geq 0$, consider the almost
complex structure $J_\alpha$ on $\C^2$ given by the matrix 
$$
J_\alpha = \left( 
\begin{array}{cc} J_0 & 0 \\
A & J_0 \end{array}\right), 
\quad \text{\rm with } 
A = \left( \begin{array}{cc} \alpha & 0 \\
0 & -\alpha \end{array}\right) 
\quad \text{\rm and }
J_0 = \left( \begin{array}{cc}  0 & -1 \\
1 & 0 \end{array}\right) .
$$
If $\alpha$ ranges in some interval $0 \leq \alpha \leq \alpha_0$,
then all the $J_\alpha$ are tamed by the symplectic form
$\omega = \alpha_0^2 \omega_0 \oplus \omega_0$ on $\C^2=\C \oplus \C$.

Now for the standard split complex structure $J_0$ on $\C^2$, the
moduli space $\cM(a,J_0)$ is clearly non-empty, since for any $z_1,z_2
\in S^1$ the map $u(z) = (z_1z,z_2)$ defines an element in it. 

In general, given any $J_\alpha$-holomorphic map $u$ in the class $a$,
note that the projection $u_1$ onto the first $\C$-factor is
holomorphic in the usual sense, and since it has degree 1 it will be a
biholomorphism. So, precomposing with a suitable element of
$\Aut(D,1)$, we may assume that this projection $u_1$ is just a rotation by 
$u_1(1)$.  A short computation shows that in this case for $u$ to be
$J_\alpha$-holomorphic, it is necessary and sufficient that
$$
\p_s u_2 + J_0 \p_t u_2 = \alpha \p_tu_1.
$$
Now if $u_2:(D,\p D) \to (\C,S^1)$ is a solution of this equation,
then 
\begin{align*}
\alpha = |\alpha \p_t u_1|
&= \frac 1 \pi \left| \int_D (\p_s u_2 + J_0\p_t u_2) dsdt \right|\\
&= \frac 1 \pi \left| \int_D d(u_2 dt - J_0 u_2ds) \right|\\
&= \frac 1 \pi \left| \int_0^{2\pi} (\cos(\theta)+ \sin(\theta) J_0)
  u_2(e^{i\theta}) d\theta \right|\\
&\leq \frac 1 \pi \int_0^{2\pi} |u_2(e^{i\theta})| d\theta = 2.
\end{align*}
So for $\alpha>2$ the moduli space $\cM(a,J_\alpha)$ is empty. 
\end{example}
The phenomenon discussed in this example is a crucial ingredient in the proof below of Gromov's Theorem~\ref{thm:gromov-notexact}.

A second, related complication is the failure of compactness, a simple
instance of which is described in Example~\ref{ex:noncompact} below.
Both of these apparent problems can be overcome by considering the
collection of all $\{\cM(a,J)\}_{a\in \pi_2(\C,L)}$ at once. In the
next subsection, I briefly discuss the compactness issue.

\subsection*{The compactness theorem for disks}

Gromov's compactness theorem, in its modern formulation, asserts that
every sequence of holomorphic curves of fixed topology and uniformly
bounded energy, whose images lie within a compact subset of the target, has a
subsequence converging in a suitable sense to a limiting ``stable
curve''. A proof in text book form for holomorphic spheres was given by McDuff and Salamon~\cite{MS:04}, and an exposition of Gromov's proof for the higher genus case was written up by Hummel~\cite{Hu:97}. The case of curves with boundary is for example discussed by Liu as part of her Ph.D. thesis~\cite{Liu:02}. For holomorphic disks with Lagrangian boundary conditions the precise statement and proof were worked out in U.~Frauenfelder's diploma thesis, published as \cite{Fr:08}.

Here, I will describe the statement under the simplifying assumption
that the target symplectic manifold $(X,\omega)$ is exact,
i.e. $\omega=d\lambda$, as in the case of my main example
$X=\C^n$. This assumption in particular implies that there are no
non-constant holomorphic spheres in $X$, and so the possible limiting
configurations are much more restricted than in the general
case. Basically, all the stable curves arising as Gromov limits will be
stable trees of disks, as explained presently.

Recall that a {\em tree} is a (finite) set $T$ together with an edge
relation $E \subset T \times T$ satisfying the following conditions:

\medskip

{\bf (symmetric)} If $\alpha E \beta$ then $\beta E \alpha$.

\medskip

{\bf (antireflexive)} If $\alpha E \beta$ then $\alpha \neq \beta$.

\medskip

{\bf (connected)} For all $\alpha,\beta \in T$ with $\alpha \neq
  \beta$ there exist $\gamma_0, \dots,\gamma_m \in T$ with $\gamma_0=
  \alpha$, $\gamma_m=\beta$ and such that $\gamma_i E \gamma_{i+1}$
  for all $0 \leq i \leq m-1$.

\medskip

{\bf (no cycles)} If $\gamma_0, \dots, \gamma_m \in T$ satisfy
  $\gamma_i E \gamma_{i+1}$ and $\gamma_i \neq \gamma_{i+2}$ for all
  $i$ then $\gamma_0 \neq \gamma_m$.

\medskip

One usually draws trees by drawing the corresponding 1-dimensional CW-complexes, which have one vertex for each $\alpha \in T$ and one (unoriented) 1-cell connecting $\alpha$ and $\beta$ whenever $\alpha E \beta$ (and $\beta E \alpha$).

It follows from these axioms that the CW-complex corresponding to a tree in this way is connected and contractible. Moreover, deleting an edge connecting two vertices $\alpha,\beta \in T$ splits the CW-complex into two connected components, and we denote the subset of vertices in the component of $\beta$ by $T_{\alpha\beta}$.
\begin{figure}[h]
 \labellist
  \small\hair 2pt
  \pinlabel $\alpha$ [br] at 32 45
  \pinlabel $\beta$ [l] at 86 39
  \pinlabel $T_{\alpha\beta}$ [tl] at 133 19
 \endlabellist
  \centering
  \includegraphics[scale=.73]{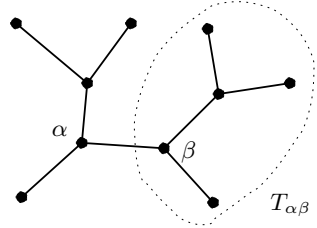}
 \caption{Two vertices $\alpha$ and $\beta$, and the corresponding subtree $T_{\alpha\beta}$.}
 \label{fig.tree}
\end{figure}

A map $f:(T,E) \to (T',E')$ is called a {\em tree homomorphism} if
for each $\alpha' \in T'$ the preimage $f^{-1}(\alpha')$ is a tree,
and moreover $\alpha E \beta$ implies that either $f(\alpha) E'
f(\beta)$ or $f(\alpha)=f(\beta)$. A bijective tree homomorphism 
is called {\em tree isomorphism}.
\begin{figure}[h]
 \labellist
  \small\hair 2pt
  \pinlabel $\alpha$ [l] at 480 311
  \pinlabel $\beta$ [l] at 560 284
  \pinlabel $\gamma$ [l] at 639 269
  \pinlabel $\delta$ [l] at 565 127
  \pinlabel $f^{-1}(\alpha)$ [l] at 76 347
  \pinlabel $f^{-1}(\beta)$ [l] at 297 270
  \pinlabel $f^{-1}(\gamma)$ [l] at 303 182
  \pinlabel $f^{-1}(\delta)$ [l] at 327 82
 \endlabellist
  \centering
  \includegraphics[scale=.53]{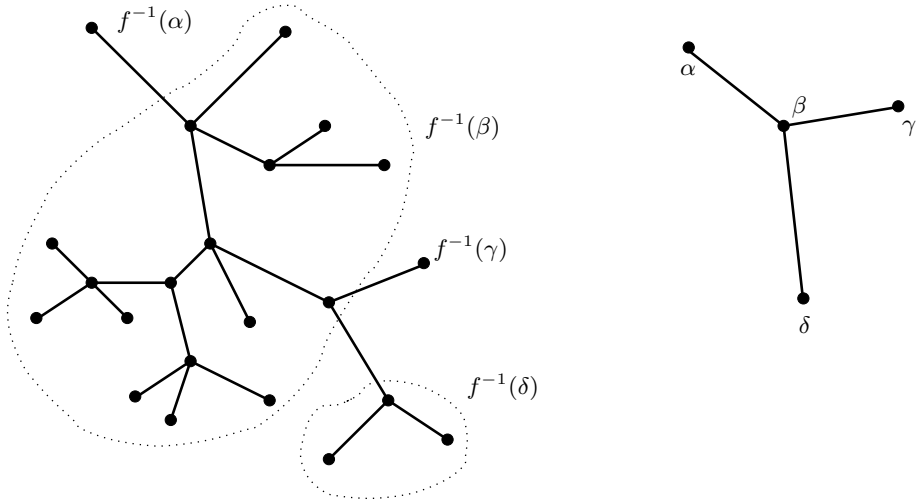}
 \caption{An example of a tree morphism.}
 \label{fig.treemorphism}
\end{figure}

Next we fix an exact symplectic manifold $(X,\omega)$ with a compact
Lagrangian submanifold $L \subset X$ and an almost complex structure
$J$ tamed by $\omega$. Given this data, we define a {\em stable tree
of holomorphic disks with one boundary marked point} to be a tuple
$$
(\mathbf u, \mathbf z) = (\{ u_\alpha\}_{\alpha \in T},
\{z_{\alpha\beta}\}_{\alpha E \beta}, \{\alpha_0,z_0\})
$$
modelled over a tree $(T,E)$, consisting of a collection of
holomorphic maps $u_\alpha:(D,\p D) \to (X,L)$ indexed by $\alpha \in
T$, a collection of nodal points $z_{\alpha\beta} \in S^1$ indexed by
directed edges $\alpha E \beta$, and a  marked point $z_0 \in \p D$
labelled by $\alpha_0 \in T$, subject to the following conditions:
\begin{enumerate}
\item If $\alpha E \beta$ then
  $u_\alpha(z_{\alpha\beta})=u_\beta(z_{\beta\alpha})$. 
\item For each $\alpha \in T$, the {\em special boundary points
  associated with the vertex $\alpha$}, namely the points
  $z_{\alpha\beta}$ for different $\beta \in T$ with $\alpha E \beta$,
  together with $z_0$ for $\alpha=\alpha_0$, are all pairwise distinct. 
\item If $u_\alpha$ is constant, then the cardinality of the set
  $Y_\alpha = \{ z_{\alpha\beta}: \alpha E \beta\} \cup \{z_0:
  \alpha_0=\alpha\}$ of all special points is at least 3.
\end{enumerate}
As usual, two such stable trees of disks $(\mathbf u,\mathbf z)$ and 
$(\mathbf u',\mathbf z')$ modelled on trees $T$ and $T'$,
respectively, are called {\em equivalent}
if there exist a tree isomorphism $f:T \to T'$ and a collection of
M\"obius transformations $\{\phi_\alpha\}_{\alpha \in T}$ such that
$$
z_{f(\alpha)f(\beta)} = \phi_\alpha(z_{\alpha\beta}), \quad
\alpha_0'=f(\alpha_0), \quad z_0'=\phi_{\alpha_0}(z_0), \text{ \rm and }\quad u'_{f(\alpha)}\phi_\alpha=u_\alpha
$$
for all $\alpha, \beta \in T$ with $\alpha E \beta$.

Define the {\em energy} of a stable tree of holomorphic disks as
the sum
$$
E(\mathbf u) := \sum_{\alpha \in T} E(u_\alpha).
$$

Now we can give the relevant notion of convergence.

\begin{definition}
Let $(X,\omega)$ be an exact symplectic manifold, let $L \subset X$ be a
compact Lagrangian submanifold, and let $J$ be an $\omega$-tame almost
complex structure on $X$. A sequence $u^\nu:(D,\p D) \to (X,L)$ of
holomorphic disks with one boundary marked point $z_0^\nu$ is said to {\em
Gromov converge} to a stable tree of disks $(\mathbf u, \mathbf z)$
with one boundary marked point if there exist sequences of elements
$\phi_\alpha^\nu \in PSL(2,\R)$ indexed by $\alpha \in T$ such that
the following statements hold:
\begin{enumerate}
\item The sequence $\phi^\nu_{\alpha_0}(z^\nu_0)$   converges to $z_0$.
\item For every $\alpha \in T$, the sequence of maps $u^\nu \circ
  \phi_\alpha^\nu$ converges to $u_\alpha$ uniformly on compact
  subsets of $D \setminus \{ z_{\alpha\beta}: \alpha E \beta\}$.
\item If $\alpha E \beta$ then $(\phi_\alpha^\nu)^{-1}\circ
  \phi_\beta^\nu$ converges to $z_{\alpha\beta}$ uniformly on compact
  subsets of $D \setminus \{z_{\beta\alpha}\}$.
\item If $\alpha E \beta$ then
\begin{equation}
\sum_{\gamma \in T_{\alpha\beta}} E(u_\gamma) = \lim_{\varepsilon \to
  0}\lim_{\nu \to \infty} E(
u^\nu|_{\phi_\alpha^\nu(B_\eps(z_{\alpha\beta}))}).
\end{equation}
\end{enumerate}
\end{definition}
Intuitively, one can imagine the M\"obius transformations
$\phi^\nu_\alpha$ as microscopes, focusing on subregions in the
domain to detect some piece of the limiting map. Item (3) in the
definition then says that different microscopes really capture
different phenomena. Item (4) is a way of phrasing that the limit
captures all the essential pieces. In particular, together with (2) it
implies that 
$$
E(\mathbf u) = \lim_{\nu\to \infty} E(u^\nu).
$$
Now Gromov's compactness theorem for disks can be stated as follows.

\begin{theorem}[{\bf Gromov compactness for holomorphic disks}]
Let $(X,\omega)$ be an exact symplectic manifold, let $L \subset X$ be a
compact Lagrangian submanifold, and let $J$ be an $\omega$-tame almost
complex structure on $X$. Suppose $u^\nu:(D,\p D) \to (X,L)$ is a
sequence of holomorphic disks with bounded energy such that all images
are contained in some compact subset of $X$. Then $u^\nu$ has a Gromov
convergent subsequence. 

Moreover, for a convergent sequence $u^\nu$ the limit is unique up to
equivalence. 
\end{theorem}

A careful exposition of the proof of this theorem, without the
simplifying assumption that $X$ is exact, can be found in
\cite{Fr:08}. 
Rather than going into details here, I will illustrate the 
phenomenon by considering a specific example.

\begin{example}\label{ex:noncompact}
%
The situation is interesting already for holomorphic disks with
boundary on $S^1 \subset \C$ with respect to the standard complex
structure $J_0$. By the maximum principle, these are
necessarily maps $D \to D$. They exist for all non-negative 
degrees, and the degree $d$ of the map equals the degree of its
restriction to the boundary circle. It turns out that the Maslov index
of the class $d \in \pi_1(S^1)$ is $2d$. Applying the general index
formula \eqref{eq:index}, one finds that 
$$
\dim \cM(d,J_0) = n-2 + 2d = 2d-1.  
$$

For $d=1$, this index equals 1. Indeed, the group of
holomorphic degree 1 maps, i.e. holomorphic automorphisms of $D$ is
3-dimensional (it can be identified with $\operatorname{PSL}(2,
\R)$). In fact, any element of this group can be written uniquely as a
composition $\varphi=\varphi_2 \circ \varphi_1$, where $\varphi_2$ is
a rotation given by $\varphi_2(z):= \varphi(1)z$ and
$\varphi_1=\varphi_2^{-1}\circ \varphi \in \Aut(D,1)$ is a map fixing
$1\in S^1=\del D$. Since we only consider the moduli space of maps up
to the equivalence relation of precomposing with an element of
$\Aut(D,1)$, the moduli space can be identified with $S^1$ by
recording the image of $1\in S^1$ under any map in the equivalence
class. In particular, the moduli space $\cM(1,J_0)$ is compact. 

Next consider a map of higher degree $d>1$. Again any such map can be
written as a composition $\varphi=\varphi_2 \circ \varphi_1$, where
$\varphi_1$ is a product (not composition!) of $d$ maps of degree $1$,
each fixing $1 \in S^1$, and $\varphi_2$ is a rotation. The map
$\varphi_1$ is characterized completely in terms of its zeros, counted
with multiplicities. In fact, if these zeros are 
$z_1$, \dots, $z_d$, then we have
$$
\varphi_1(z) = \prod e^{-i\theta_j} \frac {z-z_j}{-\bar{z_j}z+1},
\quad\text{\rm with }\theta_j= \arg(\frac{1-z_j}{1-\bar{z_j}}).
$$
By precomposing with an appropriate $\psi \in \Aut(D,1)$, we may
always arrange that one of the zeros, say $z_d$, equals 0. The
coordinates of the other zeros give $2(d-1)$ free local parameters for
$\varphi_1$, and together with the rotation parameter for $\varphi_2$
we get $2d-1$ as predicted by the dimension formula above.

The moduli space $\cM(d,J_0)$ of degree $d>1$ self-maps of the disk $D$ is
noncompact. In fact, consider representative maps $\varphi^{(n)}$ of a
sequence of points in the moduli space and orderings of the zeros
$z_j^{(n)}$ of $\varphi^{(n)}$ such that $z^{(n)}_d=0$ (as above, this
can be arranged by precomposing a given representative with some
element of $\Aut(D,1)$, which does not change 
the equivalence class). After passing to a subsequence, we get
convergence of the $z^{(n)}_j$ to some limiting $z^\infty_j\in D$ 
for each $j=1,\dots,d$. The formula above still makes sense in the
limit, but the ``zeros'' $z^\infty_j$ which lie on the circle $S^1$
contribute a trivial factor of $1$ to the product. So if there are $0<d'
<d$ of these ``phantom'' zeros, the naive limiting map
$\varphi^\infty$ will have degree $d-d'$ and so it is not an element
of $\cM(d)$.  

Notice that in the above discussion, we made two arbitrary choices: a
choice of ordering $z^{(n)}_j$ of the zeros of $\varphi^{(n)}$, and
the choice to always reparametrize so that $z^{(n)}_d=0$. Suppose for
definiteness that with these choices we have $z^\infty_1 \in
S^1$. Then  there are unique maps $\psi^{(n)} \in \Aut(D,1)$ such that
$\psi^{(n)}(z^{(n)}_1)=0$, and so we get a different sequence of
representatives $\varphi^{(n)} \circ \psi^{(n)}$ of the 
same divergent sequence of points in the moduli space $\cM(d)$. 
Just as above we get a, generally different, limiting map
$\varphi^{\infty,*}$ for a suitable subsequence. Note that in this
reparametrization, we will have $z^\infty_d\in S^1$, and so
$\varphi^{\infty,*}$ has degree $d''<d$.  

If we would analyse the situation fully, we would recover, for
a suitable subsequence, the existence of finitely many sequences of
M\"obius transformations $\psi^{(n)}_\alpha$, such that the
reparametrized maps $\varphi^{(n)} \circ \psi^{(n)}_\alpha$ converge
to some limiting map $\varphi_\alpha$ of degree $d_\alpha>0$ in such a
way that $\sum d_\alpha=d$. In addition, these maps fit together and
form a disk tree as described in the compactness theorem above.
\end{example}

\subsection*{The compactified moduli space}

Very roughly, the compactness theorem asserts that one can compactify
a given space $\cM(a,J)$ by adding pieces built out of moduli spaces
$\cM(b,J)$ with $0<E(b)<E(a)$. This compactification is often denoted by
$\overline{\cM}(a,J)$. It admits an obvious stratification, where the
stratum of codimension $k$ corresponds to stable trees of disks
modelled on trees with exactly $k$ (unoriented) edges. Indeed, the
heuristic dimension count proceeds as follows. Denote by $r_\alpha$
the number of special points on the disk associated to $\alpha \in T$
and by $a_\alpha$ its relative homotopy class. Note that $\sum
r_\alpha = 2k+1$, where $k$ is the number of edges of the tree, since
we had one marked point to start with and each edge gives rise to two
nodal points. The formal dimension of the moduli space of disks
associated to the vertex $\alpha\in T$ is
$$
n-3 +r_\alpha +\mu(\alpha).
$$
Requiring that the nodal points corresponding to an edge in $T$ are
mapped to the same point in $L$ gives $n=\dim L$ constraints. Putting
these together, we find that the formal total dimension equals
$$
(k+1)(n-3) + 2k+1 + \mu(a) -kn = n-2+ \mu(a) -k.
$$
{\em If}, for a given $J$, all the moduli spaces appearing in the
compactification $\overline{\cM}(a,J)$ were transversely cut out, one
could hope to prove a {\em gluing theorem}, asserting that in fact the
compactified moduli space is a manifold with boundary and corners. 

This is generally too much to ask. In \cite{FOOO}, Fukaya, Oh, Ohta
and Ono describe a procedure to put a so-called {\em Kuranishi
  structure} on the compactified moduli spaces. Without going into
details, this roughly means that these spaces admit fundamental chains
{\em that make them function as if they were manifolds with corners}. 
Theorem~\ref{thm:blackbox} below should be understood in this sense.

Presumably, the ongoing polyfold project of Hofer, Wysocki and Zehnder
(cf. \cite{Ho:06}) will eventually lead to an alternative approach to
the problem of putting enough structure on the compactified moduli space
to prove a statement like Theorem~\ref{thm:blackbox}.

\section{Gromov's Theorem and Fukaya's refinement}
\label{sec:main}

\subsection*{No exact Lagrangian submanifolds in $\C^n$}

As already mentioned, it is instructive to review the proof for the
following well-known theorem of Gromov. 

\begin{theorem}[Gromov, 1985]\label{thm:gromov-notexact}
If a compact manifold $L$ admits a Lagrangian embedding into $\C^n$,
then $H^1(L;\R) \neq 0$.
\end{theorem}

\begin{proof}(Sketch)
I sketch the proof of this theorem given in \cite[Section
9.2]{MS:04}, slightly rephrasing the end of the argument in order to
make the relation to the following discussion even more apparent.

Fix a Lagrangian embedding $L \subset \C^n$, and choose a vector $a\in
\C^n$ with $\|a\| \geq 2\sup_{z\in L} \|z\|$. Consider the set $\cH
\subset C^\infty([0,1] \times D \times \C^n)$ of Hamiltonian functions
such that
$$
H_{s,t}^0(z)=0, \quad H_{s,t}^1(z) = \langle a, z \rangle.
$$
The idea is to consider, for a fixed $H \in \cH$, the moduli space
$\cN$ of maps $u:(D^2,\del D^2) \to (\C^n, L)$ satisfying the
following conditions: 
\begin{itemize}
\item $\p_s u + J_0 \p_t u = \nabla H^{\lambda}_{s,t}(u)$ for some
  $\lambda \in [0,1]$, and
\item the relative homotopy class $[u] \in \pi_2(\C^n,L)$ vanishes.
\end{itemize}
So for $\lambda=0$, we are considering holomorphic disks with boundary on
the Lagrangian $L$, and since the relative homotopy class vanishes, these
are precisely the constant maps. Note that we do not divide out any 
automorphisms here, since for positive $\lambda$ these have no reason to 
preserve the solution space to the equation.
One can prove that for fixed small
$\lambda>0$, there is still a compact $n$-dimensional family of
solutions. In fact, for generic choice of the Hamiltonian $H\in \cH$,
standard transversality techniques show that $\cN$ is a smooth
$n+1$-dimensional manifold, whose boundary consists of those elements
with $\lambda \in \{0,1\}$. 

On the other hand, a straightforward computation as in
Example~\ref{ex:nosol} shows that, for our
choice of $a$, there are no solutions to the equation with
$\lambda=1$. So if $\cN$ was compact, it would give a
smooth cobordism from $L$ to the empty set. Now consider the evaluation map 
$$
\mathrm ev: \mathcal N \to \Lambda L, \quad u \mapsto u|_{\del D^2}.
$$
From what we said above, the boundary of this $(n+1)$-chain in the free loop
space of $L$ is the cycle of constant loops $[L] \in C_n(\Lambda L)$. Since
this cycle is nontrivial in homology, $\cN$ cannot be compact. 

It follows from elliptic regularity theory (cf. \cite[Theorem
4.1.1]{MS:04}) that if compactness fails, there is a 
sequence $u_n \in \mathcal N$ such that $|du_n|_\infty \to \infty$ as
$n \to \infty$. Appropriately rescaling such a sequence and applying
removal of singularities as in \cite[section 4.2]{MS:04}, one finds
either a nonconstant holomorphic sphere or a nonconstant holomorphic
disk with boundary on $L$. Since $\C^n$ does not contain nonconstant
holomorphic spheres (such spheres would be contractible and have positive 
energy ,contradicting Stokes' theorem), the only possibility is the 
existence of some nonconstant holomorphic disk $v:(D^2,\del D^2) \to (\C^n,L)$.

Now the standard symplectic form $\omega= \sum_{j=1}^n dx_j \wedge dy_j$
is positive on all complex lines in $\C^n$, and hence on all the
tangent planes to the image of $v$, so we have
$$
\int_{D^2} v^*(\omega)>0.
$$
Moreover, $\omega= d\eta$, where e.g. $\eta=\sum_{j=1}^n x_j dy_j$,
and so Stokes' theorem implies that 
$$
\int_{S^1} v^*(\eta)>0.
$$
On the other hand, the Lagrangian condition states that $\omega=d\eta$
vanishes pointwise when restricted to $L$. Combining these
observations, it follows that $\eta|_L$ is a closed 1-form
representing a nonzero class in $H^1(L;\R)$, and this proves the theorem. 
\end{proof}

\subsection*{The technical outcome of analysing holomorphic disks}

What was the essence of the proof of Gromov's theorem? Basically, the
point is that the space $\cN$ has a single boundary component,
corresponding to the space of constant disks, and hence for
topological reasons it cannot be compact. Analysing the breakdown of
compactness, we found holomorphic disks.  

The elements of $\cN$ do not appear to be holomorphic curves, due
to the nonzero right hand side $\nabla H^\lambda_{s,t}$ of the
equation. However, they can in fact be viewed as holomorphic 
maps into $\C^n \times D$ with respect to a family of almost complex
structures which have an off-diagonal term built out of this right
hand side, projecting holomorphically and with degree 1 to the
disk. This basic phenomenon was already present in Example~\ref{ex:nosol}. 

The stable map compactification in this particular case is given by
bubble trees of disks with exactly one main component, satisfying the
original equation, and all other bubbles being strictly holomorphic
(in the graph picture just mentioned, each of these has constant
projection to the disk). Fukaya's insight was to see that the new
boundary can be described in terms of string topology operations.  

Namely, for each $a\in \pi_2(\C^n,L)$, consider the compactified
moduli space $\cMbar(a):=\cMbar(a,J_0)$ of holomorphic disks in the
relative homotopy class $a$. 
Similarly, denote by $\cN(a)$ the space of solutions $u:(D,\del
D) \to (\C^n, L)$ to the equation
$$
\p_s u + J_0\p_tu= \nabla H^\lambda_{s,t} \quad \text{\rm for some
}\lambda\in[0,1] 
$$
in the relative homotopy class $a\in \pi_2(\C^n,L)$. Pretending as
always that transversality holds, this space is a manifold of dimension
$$
\dim \cN(a) = n + 1 + \mu(a),
$$
and we denote its compactification by $\overline{\cN}(a)$. 

Assuming the analysis can be made to work, both $\cMbar(a)$ and
$\overline{\cN}(a)$ can be thought of as chains 
on $\Lambda L$, simply by associating to each map of the disk its
restriction to the boundary circle. In fact, for $\cMbar$ there is a
slight ambiguity, since its elements are only well-defined up to
precomposition by $\varphi\in \Aut(D,1)$. But one can easily get
around this point, for example by replacing the actual map by a
parametrization proportional to arc length. 

To arrive at a clean statement, I will introduce some further
notation. Suppose we are given a suitable model $C_*(\Lambda L)$ for
the chains on the free loop space of $L$ with coefficients in $\Q$,
such that for each $a \in \pi_2(\C^n,L)$ the compactified spaces
$\cMbar(a)$ and $\overline{\cN}(a)$, with their respective 
evaluation maps to $\Lambda L$, define elements in it.
Note that $\Lambda L$ is a disjoint union over its connected
components $\Lambda_\gamma L$, which can be identified with conjugacy
classes $\gamma$ of elements of $\pi_1(L)$. It is convenient to
introduce a new complex $\cC$ whose underlying vector space is
$C_*(\Lambda L)$, but with grading shifted according to the Maslov
index, i.e.~an element in $C_k(\Lambda_\gamma L)$ will have degree
$k-\mu(\gamma)$ in $\cC$. 

The complex $\cC$ comes with a filtration by the symplectic area as
follows. It is shown in \cite[Prop.~4.1.4]{MS:04} that the infimum 
\begin{equation}\label{eq:minergy}
\hbar := \inf \{E(u) \mid u:(D,\del D) \to (\C^n,L) \text{ \rm nonconstant
  and holomorphic}\} 
\end{equation}
of the symplectic energy is strictly positive. Since each loop on $L$ bounds a disk in $\C^n$, we can view the energy as a map  $E:\pi_0(\Lambda L) \to \R$.
The energy of $c=\sum c_i\in \cC$ is now defined as $E(\sum c_i):= \sum E(c_i)$, where $E(c_i)$ is the energy of the free homotopy class of the loops parametrized by the chain $c_i$ (assumed to have connected domain of definition). Then the filtration
$\{\cF^k\}_{k \in \Z}$ on $\cC$ is given by  
$$
\cF^k := \{ c \in \cC: E(c) \geq k\hbar \}, \quad k \in \Z.
$$

Now consider the completion $\widehat{\cC}$ of $\cC$ with respect to
this filtration. This means that an element in $\widehat{\cC}$ will be a
possibly infinite sum $c= \sum c_i$ of chains $c_i\in \cC$, provided
that for each $k \in \Z$ there are only a finite number of summands
satisfying $c_i \not\in \cF^k$.  

With this definition, Gromov compactness and our grading convention
imply that  
$$
\cMbar := \sum_{a \in \pi_2(\C^n,L)\setminus \{0\}} \cMbar(a) 
$$
is a well-defined element of $\widehat{\cC}$ of degree $n-2$. In fact,
it is contained in the submodule $\cF^1\subset \widehat\cC$ of chains with
strictly positive area. Similarly,  
$$
\overline{\cN} := \sum_{a \in \pi_2(\C^n,L)} \overline{\cN}(a) 
$$
is an element of $\widehat{\cC}$ of degree $n+1$, as follows by applying 
the analogue of \eqref{eq:minergy} to the graph of elements of $\cN(a)$ in 
$D \times \C^n$.

As explained at the end of section 2, the compactification of $\cMbar(a)$ is 
obtained by adding lower dimensional strata built as fiber products of other 
such moduli spaces along evaluation maps. In particular, the codimension 1 
pieces $\lambda_1(\cMbar(a))$ are build from configurations of two holomorphic 
disks for which suitable boundary points are mapped to the same point in $L$. 
A more careful analysis reveals that, on the level of boundary values of the 
holomorphic maps, these configurations correspond to loop brackets 
$\lambda_2(\cMbar(a_1),\cMbar(a_2))$ with $a_1+a_2=a$. 
\begin{example}
\label{ex:2torus}
The mechanism just described can be seen in Example~\ref{ex:noncompact}, but maybe it 
is slightly easier to visualize for the standard Lagrangian torus 
$T^2=S^1 \times S^1 \subset \C^2$. Any class $a\in \pi_2(\C^2,T^2)$ is characterized by two integers $(d_1,d_2)$ giving the degrees of the projections to the two coordinate disks. For a moduli space $\cMbar((d_1,d_2),J_0)$ with respect to the standard complex structure $J_0$ on $\C^2$ to be nontrivial we need $d_j \ge 0$. Leaving aside the constant maps, the simplest moduli spaces $\cM((1,0),J_0)$ and $\cM((0,1),J_0)$ are compact, and in fact one can identify both with $T^2$. Indeed, the equivalence classes of the maps $u_{z_1,z_2}:D^2 \to \C^2$ given by $u_{z_1,z_2}(z)=(z_1z,z_2)$ for 
$z_1,z_2\in S^1$ represent all elements in $\cM((1,0),J_0)$, and similarly the maps $v_{z_1,z_2}(z)=(z_1,z_2z)$ represent all elements in $\cM((0,1),J_0)$. Note that because of the symmetries in the problem, in this particularly simple example the evaluation maps at $1$ are submersions, so that the geometric definition of the loop bracket $\lambda_2(\cM((0,1)),\cM((1,0)))$ can be used.

To illustrate the discussion above, we want to argue that 
\begin{align*}
\p \cMbar((1,1),J_0)&= \lambda_2(\cM((0,1)),\cM((1,0)))\\
&= \frac 12 \Big[\lambda_2(\cM((0,1)),\cM((1,0)))+\lambda_2(\cM((1,0)),\cM((0,1)))\Big].
\end{align*} 
Every element of the space $\widetilde{\cM}((1,1),J_0)$ is a map $u:D^2 \to \C^2$ of the form $u(z)=(z_1\phi_1(z),z_2\phi_2(z))$ with $z_j\in S^1$ and $\phi_j\in \Aut(D,1)$. The space $\cM((1,1),J_0)$ is obtained as the quotient by the diagonal action of $\Aut(D,1)$, so the equivalence class of $u$ as above is alternatively represented by both $u'(z)=(z_1z, z_2\psi(z))$ or $u''(z)=(z_1 \psi^{-1}(z),z_2z)$, where $\psi=\phi_2\circ\phi_1^{-1}$ is uniquely associated with the equivalence class of $u$.

Now consider a sequence $u_n \in \widetilde{\cM}((1,1),J_0)$ with $z_1$ and $z_2$ fixed but $\phi_{1,n}$ and $\phi_{2,n}$ varying. Assume that the  projection of the sequence to $\cM((1,1),J_0)$ leaves every compact subset, meaning that the corresponding sequence $\psi_n\in \Aut(D,1)$ in the above notation does the same. Elementary considerations now show that for a suitable subsequence $\psi_{n_k}$ there will be a point $w\in S^1$ such that 
\begin{enumerate}
\item[(i)] $\psi_{n_k} \longrightarrow w$ uniformly on compact subsets of $D \setminus \{1\}$ and $\psi_{n_k}^{-1} \longrightarrow 1$ uniformly on compact subsets of $D \setminus \{w\}$, or
\item[(ii)] $\psi_{n_k} \longrightarrow 1$ uniformly on compact subsets of $D \setminus \{w\}$ and $\psi_{n_k}^{-1} \longrightarrow w$ uniformly on compact subsets of $D \setminus \{1\}$.
\end{enumerate}
In both cases, the corresponding subsequences $u'_{n_k}$ and $u''_{n_k}$ converge to elements  $u'_\infty \in \cM((1,0),J_0)$ and $u''_\infty \in \cM((0,1),J_0)$, respectively, and the pair $(u'_\infty,u''_\infty)$ represents a boundary point of $\cMbar((1,1),J_0)$. One checks that as $z_1$ and $z_2$ and the sequence $\psi_{n_k}$ vary, one obtains all boundary points from this construction.
In case (i), $u''_\infty(1)=(z_1,z_2)$ and $u'_\infty(1)=(z_1,z_2w)$, which is the unique intersection point of $u'_\infty(S^1)$ and $u''_\infty(S^1)$, and in case (ii) the roles are reversed. 
In particular, the boundary loops of the two limit disks concatenate to represent points in the loop bracket $\lambda_2(\cM((0,1),J_0),\cM((1,0),J_0))$.
\end{example}

We now return to the general discussion. Similarly to the case of $\cMbar$, the codimension 1 stratum for $\overline{\cN}(a)$ corresponds to 
stable maps consisting of one component satisfying the perturbed equation and 
one holomorphic disk, and so it is described by the loop brackets of the form 
$\lambda_2(\overline{\cN}(a_1),\cMbar(a_2))$. Depending on the precise 
technical implementation, the gluing along lower dimensional boundary strata 
might actually introduce more terms, corresponding to higher operations.

The main technical assertions which should come out of such an implementation 
can be formulated as the following theorem. To get a cleaner statement, I
have chosen to state it in slightly stronger form than is strictly
necessary. The concept of a filtered L$_\infty$ algebra which appears in the
statement is discussed in detail in the following section, where I
also give some algebraic perspective on the equations
\eqref{eq:fukaya1} and \eqref{eq:fukaya2}. 
\begin{theorem}\label{thm:blackbox}
Let $L \subset \C^n$ be a closed, oriented, spin Lagrangian
submanifold.

Then on the filtered, degree-shifted chain complex $\widehat{\cC}$
associated to a suitable chain model $C_*(\Lambda L)$ for the free
loop space $\Lambda L$ there exists a filtered L$_\infty$-algebra structure 
$\{\lambda_k\}_{k\geq 1}$ of degree $1-n$, whose bracket on
homology coincides with the loop bracket of string topology, and such
that
\begin{enumerate}
\item the union of moduli spaces $\cMbar$ gives rise to an element
  $\alpha \in \widehat{\cC}_+$ of degree $n-2$ satisfying 
\begin{equation}\label{eq:fukaya1}
\sum_{k=1}^\infty (-1)^{\frac {(k-1)k}{2}}\frac 1 {k!}
\lambda_k(\alpha,\cdots,\alpha) =0. 
\end{equation}
\item the union of moduli spaces $\overline{\cN}$ gives rise to an
  element $\beta \in \widehat{\cC}$ of degree $n+1$
  satisfying  
\begin{equation}\label{eq:fukaya2}
\sum_{k=1}^\infty (-1)^{\frac {(k-2)(k-1)}{2}}\frac 1 {(k-1)!}
\lambda_k(\beta,\alpha,\cdots,\alpha) = [L],
\end{equation}
where $[L]\in C_n(\Lambda L, \Q)$ denotes the chain of constant loops.
\end{enumerate}
\end{theorem}
I will treat this theorem as a black box, and deduce the main results in the
introduction from it by using abstract algebraic arguments and some
3-manifold topology.

\section{Some algebraic properties of L$_\infty$ algebras}
\label{sec:l-infty}

For a graded vector space $C = \oplus_{d\in \Z} C_d$ we denote by $C[n]$ 
the vector space with grading shifted by $n$, i.e.  $C[n]_d=C_{d+n}$.  
On the $k$-fold tensor product $C \otimes \cdots \otimes C$, we consider
two actions of the permutation group $S_k$. In the first one, a permutation 
$\rho \in S_k$ acts on some tensor product of elements $c_i\in C$ of pure 
degrees $|c_i|$ via
$$
\rho \cdot (c_1 \otimes \cdots \otimes c_k) = \eps(\rho;c_1,\dots,c_k) \cdot c_{\rho(1)} \otimes \cdots \otimes c_{\rho(k)}
$$ 
with $\eps(\rho;c_1,\dots,c_k) = (-1)^{\sum_{i<j,\rho(i)>\rho(j)} |c_i|\cdot |c_j|}$. The quotient is the $k$th symmetric power $S^kC$ of $C$, whose decomposable elements we write as $c_1\cdots c_k$. The second action is the first one twisted by the sign representation,
$$
\rho \cdot (c_1 \otimes \cdots \otimes c_k) = \operatorname{sgn}(\rho) \eps(\rho;c_1,\dots,c_k)\cdot c_{\rho(1)} \otimes \cdots \otimes c_{\rho(k)}.
$$ 
The quotient is the $k$th exterior power $\Lambda^k C$ of $C$, whose elements are usually denoted by $c_1 \wedge \cdots \wedge c_k$. With these definitions, for an element $c\in C$ of odd degree we have $c\cdot c=0$, but $c \wedge c \ne 0$.

\begin{definition}
An L$_\infty$ algebra of degree 0 consists of a graded vector space
$C$ and a sequence of multilinear operations
$$
\lambda_k: \Lambda^k C \to C, \quad k \geq 1
$$
of degree $|\lambda_k|=k-2$ satisfying the sequence of quadratic
relations 
\begin{equation}\label{eq:l-infty}
\sum_{k_1+k_2=k+1,\atop \rho\in S_{k}} \pm \frac 1
  {k_1!(k-k_1)!} \lambda_{k_2}(\lambda_{k_1}(c_{\rho(1)}, \dots,
  c_{\rho(k_1)}),c_{\rho(k_1+1)}, \dots,c_{\rho(k)})=0
\end{equation}
for each $k \geq 1$. (The signs are made explicit below.)

More generally, an L$_\infty$ algebra structure of degree $d$ on $C$ is
defined to be an L$_\infty$ structure of degree $0$ on $C[-d]$.
\end{definition}
\begin{remark}
If the vector space $C$ of an L$_\infty$ algebra of degree 0 is
concentrated in degree 0, then for degree reasons the only possibly
nontrivial operation is $\lambda_2$, and the relation for $k=3$ turns
out to be the Jacobi identity for $\lambda_2$, so we recover Lie
algebras as a special case. 
\end{remark}

\begin{remark}
If $\lambda_k=0$ for $k \geq 3$, then we recover the definition of a
dg Lie algebra. Indeed, the first relation reads $\lambda_1 \circ
\lambda_1=0$. The second relation shows that $\lambda_1$ is a
derivation of $\lambda_2$, and the third relation is again the Jacobi
identity. In general, the Jacobi identity holds ``up to homotopy''
given by $\lambda_3$, so it always holds for the induced bracket on
$H_*(C,\lambda_1)$. 
\end{remark}
To make the signs in the quadratic relations as well as other signs
below explicit, it is useful to give an alternative description. First
observe the graded linear isomorphism  
\begin{align*}
\sigma_k: (\Lambda^k C)[-k] &\to S^k(C[-1])\\
          c_1 \wedge \dots \wedge c_k & \mapsto (-1)^{\sum (k-i)|c_i|}
          c_1 \cdots c_k,
\end{align*}
where $|c_i|$ denotes the degree in $C$.
Next introduce operations $\ell_k:S^k(C[-1])\to C[-1]$ as $\ell_k=\sigma_1
\circ \lambda_k \circ \sigma_k^{-1}$ and note that with this degree
shift these are all of degree $-1$. Set $S(C[-1]):= \oplus_{k\geq 1}
S^k(C[-1])$ and observe that each of these operations can be
extended to a map $\hat\ell_k: S(C[-1]) \to S(C[-1])$ defined as
$$
\hat\ell_k(c_1 \cdots c_r) = \left \{ \begin{array}{cl}
0 & \text{\rm if } r <k\\
\sum_{\rho \in S_r} \frac {\eps(\rho,c_1,\dots,c_r)}{k!(r-k)!}
\ell_k(c_{\rho(1)} \cdots c_{\rho(k)})c_{\rho(k+1)} \cdots c_{\rho(r)}
  & \text{\rm if } r \geq k,
\end{array} \right.
$$
where $\eps(\rho,c_1,\dots,c_r)$ is the sign introduced above

Finally, one defines $\hat \ell:= \sum_{k \geq 1} \hat \ell_k:S(C[-1])
\to S(C[-1])$. Then the quadratic relations \eqref{eq:l-infty} (with the correct signs) are equivalent to the single equation
\begin{equation}\label{eq:l-infty2}
\hat \ell \circ \hat \ell = 0.
\end{equation}
The above passage from the operations $\lambda_k$
on $C$ to the operation $\hat \ell$ on $S(C[-1])$ is called the {\em bar
construction}. Conceptually, one views $S(C[-1])$ as a 
coalgebra via the comultiplication $\Delta: S(C[-1]) \to S(C[-1])
\otimes S(C[-1])$ given by 
$$
\Delta (c_1 \cdots c_r) = \sum_{r_1=1}^{r-1}\sum_{\rho\in S_r} \frac {\eps(\rho,
  c_1,\dots, c_r)}{r_1!(r-r_1)!} c_{\rho(1)} \cdots c_{\rho(r_1)}
\otimes c_{\rho(r_1+1)} \cdots c_{\rho(r)}.
$$
This map has the coassociativity property
$$
(\one \otimes \Delta) \circ \Delta = (\Delta \otimes \one) \circ
\Delta,
$$
and it also turns out to be cocommutative in the sense that
$\tau \circ \Delta = \Delta$, where $\tau:S(C[-1])\otimes S(C[-1]) \to
S(C[-1]) \otimes S(C[-1])$ is the signed permutation of the two factors.
Then $\hat\ell_k$ is the unique way to extend $\ell_k$ as a
coderivation, i.e.~as a map satisfying the co-Leibniz rule
$$
\Delta \hat\ell_k = (\hat \ell_k \otimes \one + \one \otimes \hat
\ell_k) \Delta.
$$
Conversely, one can prove that any coderivation  $D:S(C[-1]) \to S(C[-1])$ is
completely determined by its {\em linear part} $\pi_1 \circ D:S(C[-1]) \to
C[-1]$. So $\hat \ell$ is the unique coderivation
of degree $-1$ on $S(C[-1])$ such that the restriction of its linear part
to $S^k(C[-1])$ equals $\ell_k$. \newline 
It is also easy to see that the commutator
$[D_1,D_2]:= D_1 \circ D_2 - (-1)^{|D_1||D_2|} D_2 \circ D_1$ of two
homogeneous coderivations is a coderivation, and so in our example
above $\hat \ell \circ \hat \ell = \frac 1 2 [\hat \ell, \hat \ell]$
has this property. These remarks explain why the relation
\eqref{eq:l-infty2} is equivalent to the sequence of relations
\eqref{eq:l-infty}, since this sequence is obtained by restricting the
linear part of $\hat \ell \circ \hat \ell$ to $S^k(C[-1])$ for each $k
\geq 1$ (and precomposing with $\sigma_k$). 

So in summary, an L$_\infty$ structure on a graded vector space $C$ is
the same as a coderivation of square zero on the symmetric tensor
coalgebra $S(C[-1])$. 
\begin{remark}
I have adopted homological conventions here, whereas often in the
literature one finds cohomological conventions, where the $\lambda_k$
have degrees $2-k$, and in the bar construction one shifts degrees by
$1$ instead of $-1$. 
\end{remark}
\begin{definition}
Given two L$_\infty$ algebras $\cC=(C,\{\lambda_k\}_{k \geq 1})$ and
$\cC'=(C',\{\lambda'_k\}_{k \geq 1})$, a morphism from $\cC$ to $\cC'$
consists of a sequence of maps $\phi_k:\Lambda^k C \to C'$ of degrees
$|\phi_k|=k-1$ satisfying the sequence of relations
\begin{equation}
\sum_{k_1k_2=k+1} \pm \phi_{k_1} \circ \hat\lambda_{k_2} = \sum_{k_1 +
  \dots + k_r=k} \pm \frac 1 {r!} \lambda'_r \circ (\phi_{k_1} \otimes
\dots \otimes \phi_{k_r}) 
\end{equation}
for $k \geq 1$. 
\end{definition}
Again, to state the signs correctly, it is useful to pass to the
associated maps $f_k:S^k(C[-1]) \to C'[-1]$ of degree $0$ given by $f_k
= \sigma_1 \circ \phi_k \circ \sigma_k^{-1}$. Any such collection of
linear maps determines a unique morphism of coalgebras $e^f:
S(C[-1]) \to S(C'[-1])$, given by
$$
e^f(c_1 \cdots c_k) = \sum_{k_1 + \dots + k_r=k} \sum_{\rho\in S_k}
\frac {\eps(\rho,c_1,\dots,c_k)}{r!k_1!\cdots k_r!} (f_{k_1} \otimes
\cdots \otimes f_{k_r}) (c_{\rho(1)} \cdots c_{\rho(k)}).
$$
The fact that $\{\phi_k\}$ is a morphism of L$_\infty$ algebras can
now be stated equivalently (including the correct signs) as
\begin{equation}
e^f \hat \ell = \hat \ell' e^f.
\end{equation}

The first important result about L$_\infty$ algebras asserts that the
structure of an L$_\infty$ algebra can be transferred from a complex
$C$ to its homology with respect to $\lambda_1$, without the loss of
any essential information. More precisely, it is formulated as follows. 

\begin{theorem}\label{thm:HPT}
Suppose $\cC=(C,\{\lambda_k\}_{k \geq 1})$ is an L$_\infty$ algebra over a field of characterstic 0. Then there exists an L$_\infty$ algebra structure $\cH=(H_*(C,\lambda_1), \{\lambda'_k\}_{k \geq 2})$ on the homology which is homotopy equivalent to $\cC$.  
%
\end{theorem}
Here a {\it homotopy equivalence} between L$_\infty$ algebras is the
essentially obvious generalization of the classical notion. In
particular, it is an L$_\infty$ morphism which induces an isomorphism in the
homology of the underlying complexes. It is a theorem that every such
map admits a homotopy inverse. For detailed definitions and a proof of
these assertions, including the theorem, see e.g. \cite{LV}.

The construction of the homotopy equivalence starts with a linear
homotopy equivalence $\iota:H_*(C,\lambda_1) \to C$ given by choosing a
cycle in each homology class, which has a homotopy inverse $\pi:C \to
H_*(C,\lambda_1)$ given by projection along a complement of the image
of $\iota$. One sets $\phi_1=\iota$ and $\lambda'_2= \pi \circ
\lambda_2 \circ \iota \otimes \iota$, and constructs 
the higher maps $\phi_k, k \geq 2$ and operations $\lambda'_k$, $k \geq
3$ simultaneously by induction. The fact that the homologies of
the two complexes agree is used to prove that all relevant obstructions
vanish. 

\medskip

Now let $\cC=(C,\{\lambda_k\}_{k \geq 1})$ be an L$_\infty$
algebra. Suppose that $C$ is the completion of some
complex $C'$ with respect to a doubly infinite filtration $C' =
\cup_{k\in \Z} \cF'_k$ with $\cF'_k \supset \cF'_{k+1}$, so that
elements of $C$ are (possibly infinite) sums of elements of $C'$ of
the form 
$$
c= c_{-r} + \dots + c_{-1} + c_0 + c_1 + \dots, \quad c_k \in \cF'_k.
$$
Denote the induced filtration on $C$ by $\{\cF_k\}$. 
\begin{definition}
The L$_\infty$ structure on $C$ is called {\em filtered} if
$$
\lambda_k(\cF_{d_1},\dots,\cF_{d_k}) \subset \cF_{d_1 + \dots + d_k}.
$$
\end{definition}
In the following discussion, it is convenient to denote by $\bar c =
\sigma(c)$ the image of an element under the identity
map $\sigma:C \to C[-1]$ of degree $+1$.
An element $a \in \cF_1$ of degree $-1$ satisfying the equation
$$
\sum_{k \geq 1} \frac 1 {k!}\ell_k (\bar{a},\dots,\bar{a}) = 0
\quad \text{\rm in } C[-1]
$$
is called a {\em Maurer-Cartan element of $\cC$}. Since $a\in \cF_1$,
the left hand side of this equation is indeed a well-defined element
of $\cC$. Note that \eqref{eq:fukaya1} is an instance of this
equation. Also, this equation is equivalent to 
$$
\hat \ell(e^{\bar{a}})=0.
$$
Moreover, an easy calculation yields
\begin{lemma}
If $a\in C$ is a Maurer-Cartan element of $\cC$ and $b \in C$ is
arbitrary, then  
$$ 
\hat \ell (\bar{b} e^{\bar{a}}) = \sum_{k \geq 1}
\frac 1 {(k-1)!}\ell_k(\bar{b}, \bar{a}, \dots, \bar{a}) e^{\bar{a}}.
$$ 
In particular, the map $\hat\ell^{a}:C[-1] \to C[-1]$ given by
$$
\hat\ell^{a}(\bar{b}) = \hat\ell(\bar{b}e^{\bar{a}})e^{-\bar{a}} = 
\sum_{k \geq 1} \frac 1 {(k-1)!}\ell_k(\bar{b}, \bar{a}, \dots, \bar{a})
$$
is a differential.\hfill\qed
\end{lemma}
Finally, we describe what happens to equations \eqref{eq:fukaya1} and
\eqref{eq:fukaya2} under morphisms.
\begin{proposition}\label{prop:pushforward}
Suppose $\{\phi_k\}_{k \geq 1}$ is a morphism between L$_\infty$
algebras $\cC$ and $\cD$ preserving filtrations as above. 
\begin{enumerate}
\item If $a \in C$ is a Maurer-Cartan element for $\cC$, then
  $a'\in D$ with $\bar{a}'=(\sum_{k} \frac 1 {k!}f_k
  (\bar{a},\dots,\bar{a}))$ is a Maurer-Cartan element for $\cD$.
\item If $a \in C$ is a Maurer-Cartan element and $b,c \in C$ satisfy
$$
\hat \ell^\cC (\bar{b} e^{\bar{a}}) = \bar{c}e^{\bar{a}},
$$
then the elements $a'$, $b'$ and $c'$ with
\begin{align*}
\bar{a}'&= \left(\sum_k\frac 1 {k!} f_k(\bar{a},\dots,\bar{a})\right), \\ 
\bar{b}'&= \left(\sum_k\frac 1 {(k-1)!}
  f_k(\bar{b},\bar{a},\dots,\bar{a})\right)\quad  
\text{\rm and}\\ 
\bar{c}'&= \left(\sum_k \frac 1 {(k-1)!}
  f_k(\bar{c},\bar{a},\dots,\bar{a})\right) 
\end{align*}
satisfy
$$
\hat \ell^\cD (\bar{b}'e^{\bar{a}'}) = \bar{c}'e^{\bar{a}'}.
$$
\end{enumerate}
\end{proposition}
\begin{proof}
To prove the first assertion, just observe that for a Maurer-Cartan
element $a\in C$ one has
$$
0= e^f \hat \ell^\cC (e^{\bar{a}}) = \hat \ell^\cD e^f (e^{\bar{a}}) =
\hat \ell^\cD(e^{\bar{a}'} ),
$$
where the equality $e^f (e^{\bar{a}})=e^{\bar{a}'}$ follows
directly from the definitions.

To prove the second assertion, one first checks that for any elements
$\bar{x}, \bar{y} \in C[-1]$ 
$$ 
e^f(\bar{x} e^{\bar{y}}) =  \left(\sum_{k \geq 1} \frac 1 {(k-1)!} f_k(\bar{x},\bar{y},\dots, \bar{y})\right) 
e^{\sum \frac 1 {r!} f_r(\bar{y},\dots, \bar{y})}.
$$
Using this, we compute 
\begin{align*}
e^f \hat \ell^\cC (\bar{b} e^{\bar{a}}) 
 = \hat \ell^\cD e^f (\bar{b} e^{\bar{a}}) 
 = \hat \ell^\cD (\bar{b}'e^{\bar{a}'}).
\end{align*}
On the other hand, 
\begin{align*}
e^f \hat \ell^\cC (\bar{b} e^{\bar{a}}) = e^f
(\bar{c}e^{\bar{a}}) = \bar{c}'e^{\bar{a}'}. 
\end{align*}
\end{proof}

\section{The proofs of Theorem~\ref{thm:fukayamain} and
  Corollary~\ref{cor:fukaya1}} 
\label{sec:proofs}

Looking back at Theorem~\ref{thm:blackbox}, we see that it asserts
that the holomorphic disks with boundary on the Lagrangian submanifold
give rise to a Maurer-Cartan element $\alpha$ in the L$_\infty$ structure on
$\widehat{\cC}$ such that with respect to the twisted
differential the element $[L] \in \widehat{\cC}$ becomes
exact. Since $[L]$ is never exact with respect to the ordinary boundary
operator $\del = \lambda_1$, this tells us that both the
Maurer-Cartan element $\alpha$ and at least one of the operations
$\lambda_k$ with $k \geq 2$ must be nontrivial, since otherwise the twisted
differential coincides with the untwisted boundary operator $\p$. This 
observation lies at the core of Fukaya's proof
of Theorem~\ref{thm:fukayamain}. Before I discuss that, I will state
two purely topological facts that will turn out to be useful.

\begin{lemma}\label{lem:top1}
Let $\gamma:S^1 \to L$ be a loop, and denote by $Z \subset \pi_1(L)$
the centralizer of $\gamma$, i.e. the set of all elements commuting
with $\gamma$. Let $\pi:\tilde L \to L$ be a connected covering of $L$
associated to the subgroup $Z$, and let $\tilde\gamma$ be a lift of
$\gamma$. Then the projection  $\pi$ induces a homeomorphism
$\Pi:\Lambda_{\tilde\gamma} \tilde L \to \Lambda_\gamma L$ between
the components of $\tilde \gamma$ and $\gamma$ in the respective free
loop spaces.
\end{lemma}
\begin{proof}
Since $\pi:\tilde L \to L$ is a covering, any free homotopy $h:[0,1]
\times S^1 \to L$ with $h|_{\{0\}   \times S^1}=\gamma$ admits a
(unique) lift $\tilde h$ to $\tilde L$ with $\tilde h|_{\{0\} \times
  S^1}=\tilde\gamma$, and so in particular $h|_{\{1\} \times S^1}$ is
the image of $\tilde h|_{\{1\} \times S^1}$ under the map $\Pi$
induced by the projection. This proves surjectivity of $\Pi$.

To prove injectivity, assume that $\Pi(\tilde \delta_1)= \Pi(\tilde
\delta_2)=\delta$. Note that our two lifts $\tilde \delta_1$ and $\tilde \delta_2$ of $\delta$ are related by a
deck transformation, i.e. by the action of some homeomorphism $g:
\tilde L \to \tilde L$ satisfying $\pi \circ g = \pi$.
If $\tilde h_1$ is a free homotopy from $\tilde \gamma$ to $\tilde
\delta_1$, then $g \circ \tilde h_1$ is a homotopy from $g \circ
\tilde \gamma$ to $\tilde \delta_2$. Since by assumption
$\tilde\delta_2$ is also freely homotopic to $\tilde \gamma$, we
conclude that if $\Pi$ is not injective, then $\gamma$ has at least
two preimages, namely $\tilde \gamma$ and $g \circ \tilde \gamma$.

Now suppose $\tilde \gamma$ and $g \circ \tilde \gamma$ are freely
homotopic for some deck transformation $g$, so that they are both
preimages of $\gamma$ under $\Pi$. A free homotopy $\tilde h$
from $\tilde \gamma$ to $g \circ \tilde \gamma$ can be reinterpreted as a based
homotopy from $\tilde \gamma$ to $\chi * (g \circ \tilde \gamma) *
\chi^{-1}$, where $\chi= \tilde h|_{[0,1] \times \{1\}}$ is the path
travelled by the base point under the homotopy. Note that $\chi$
projects to a closed loop in $L$ representing $\hat g \in
\pi_1(L)$. In particular, the projection of the homotopy $\tilde h$
yields that  
$$
\gamma \cong \hat g \gamma \hat g^{-1} \quad \text{ \rm in } \pi_1(L).
$$
But $Z= \pi_*(\pi_1(\tilde L))$ was chosen to be the centralizer of $\gamma$ in
$\pi_1(L)$, so $\hat g \in \pi_*(\pi_1(\tilde L))$. In other words, 
this implies that $\chi$ was a closed loop and so 
$\tilde \gamma = g \circ \tilde \gamma$, i.e. any two preimages of
$\gamma$ coincide. Together with the previous
observation this shows that $\Pi$ is injective, completing the proof
of the lemma.
\end{proof}
\begin{lemma}\label{lem:top2}
In the situation of the previous lemma, assume moreover that $L$ (and
so $\tilde L$ as well) is aspherical. Then evaluation at the base point
$ev: \Lambda_{\tilde{\gamma}} \tilde L \to \tilde L$ is a homotopy
equivalence. 
\end{lemma}

\begin{proof}
The fiber of the map $ev: \Lambda_{\tilde{\gamma}} \tilde L \to \tilde
L$ at $\tilde\gamma(0)$ is the space $\Omega_{\tilde{\gamma}}\tilde L$
of loops which are based at $\tilde\gamma(0)$ and {\em freely}
homotopic to $\tilde\gamma$. As in the previous proof, we observe 
that any free homotopy between $\delta \in
\Omega_{\tilde{\gamma}}\tilde L$ and $\tilde\gamma$ can be
reinterpreted as a based homotopy between $\delta$ and $\chi *
\tilde\gamma * \chi^{-1}$. But $\tilde\gamma$ is central in
$\pi_1(\tilde L)$, so that $\chi * \tilde \gamma * \chi^{-1}$ is {\em 
  based} homotopic to $\tilde \gamma$. So we conclude that in fact
$\Omega_{\tilde{\gamma}} \tilde L$ is the component of $\tilde\gamma$
in the based loop space of $\tilde L$, which is contractible since
$\tilde L$ is aspherical. So $ev$ is a fibration with contractible
fibers, and hence a homotopy equivalence.
\end{proof}
\begin{corollary}\label{cor:top3}
If $L$ is an aspherical manifold, then every component of the free loop space
$\Lambda L$ has the homotopy type of a CW complex of dimension at most
$\dim L$. \hfill $\qed$
\end{corollary}
After these preliminaries, I come to the proof of the main theorem.
\begin{proof}(of Theorem~\ref{thm:fukayamain})
Recall that the Maurer-Cartan element $\alpha \in \widehat{\cC}$ is
built from the moduli spaces $\{\cMbar(a)\}_{a \in \pi_2(\C^n,L)}$,
which have geometric dimensions
$$
\dim \cMbar(a) = n-2 +\mu(a),
$$
and the element $\beta\in \widehat{\cC}$ is built from the spaces
$\{\overline{\cN}(a)\}_{a\in \pi_2(\C^n,L)}$ with geometric dimensions
$$
\dim \overline{\cN}(a) = n+1 + \mu(a).
$$
Denote by $\widehat{\cH}$ the homology with respect to the usual
boundary operator of $\widehat{\cC}$. Recall that we denote by $\bar
c$ the image of $c\in \widehat{\cC}$ under the degree shift
$\widehat{\cC} \to \widehat{\cC}[-1]$. According to
Theorem~\ref{thm:HPT}, the L$_\infty$ structure on $\widehat{\cC}$
pushes forward to an L$_\infty$ structure on $\widehat{\cH}$ under a
homomorphism $e^f$, and by Proposition~\ref{prop:pushforward}, this
homomorphism maps the elements $\alpha$, $\beta$ and $[L]$ in
$\widehat{\cC}$ to elements $\alpha'$, $\beta'$ and $[L]$ in
$\widehat{\cH}$ satisfying the equation
$$
\sum_{k=2}^\infty \frac 1 {(k-1)!}
\ell^\cH_k(\bar\beta',\bar\alpha',\cdots,\bar\alpha') = \overline{[L]}. 
$$
Writing
$$
\alpha' = \sum_{a\in \pi_2(\C^n,L)} \alpha'(a), \quad  \beta' =
\sum_{a\in \pi_2(\C^n,L)} \beta'(a), \
$$
the part of this equation corresponding to the trivial relative
homotopy class can be written more explicitly as
\begin{equation}\label{eq:degrees}
\sum_{k=2}^\infty \frac 1 {(k-1)!}
\sum_{a= a_1 + \dots + a_{k-1}} \ell^\cH_k(\bar\beta'(-a),\bar\alpha'(a_1),\cdots,\bar\alpha'(a_{k-1})) = \overline{[L]}. 
\end{equation}
Since the homomorphism between the L$_\infty$ structures preserves
degrees, the geometric degrees of $\alpha'(a_i)$ and $\beta'(-a)$ are
$n-2+\mu(a_i)$ and $n+1+\mu(-a)=n+1-\mu(a)$, respectively.

By the assumption that $L$ is aspherical, Corollary~\ref{cor:top3}
implies that the homology $\widehat{\cH}$ is concentrated in geometric
degrees $0 \leq d \leq n$. Combining this observation with
\eqref{eq:degrees} and the fact that the Maslov index is even for
orientable Lagrangian submanifolds $L$, we find that for 
the term
$\ell^\cH_k(\bar\beta'(-a),\bar\alpha'(a_1),\cdots,\bar\alpha'(a_{k-1}))$
to be nonzero we must have
$$
2 \leq \mu(a) \leq n+1, \quad \text{\rm and} \quad 2-n \leq \mu(a_i)
\leq 2.
$$
The first equation immediately implies that $\mu$ is not identically
zero. Moreover, if $\mu(a_i) \leq 0$ for all $i= 1 \dots,k-1$, it
follows that $\mu(a) \leq 0$, again contradicting the first
equation. Thus we conclude that some $\alpha(a_i)$ with
$\mu(a_i)=2$ must be nonzero, implying that the corresponding moduli
space is nonempty. So $a_i$ is represented by a holomorphic
disk, and hence must have positive symplectic energy.

Set $\gamma := \del (a_i) \in \pi_1(L)$ and let $C \subset \pi_1(L)$
denote the centralizer of $\gamma$. Notice that we have a short
exact sequence
$$
0 \to \Ker(\mu|_Z) \to C \stackrel{\frac 1 2 \mu}{\longrightarrow}
\Z \to 0,
$$
in which the last map admits an inverse sending $1$ to $\gamma$. It
follows that the map $\rho:\Z \times \Ker(\mu|_C) \to C$, defined by
$\rho(k,g)=\gamma^k \cdot g$, is an isomorphism ($\rho$ is indeed a group
homomorphism because $\gamma$ commutes with all elements of $C$). Since $L$ is
a $K(\pi,1)$, the covering space $\tilde L$ of $L$ with $\pi_1(\tilde
L) = C$ is a $K(\Z \times \Ker(\mu|_C),1)$, so it is homotopy
equivalent to $S^1 \times L'$ for a $K(\Ker(\mu|_C),1)$ space $L'$.

To complete the proof of the theorem, it remains to show that $L'$ is
closed or, equivalently, that $\tilde L \to L$ is a finite covering space.

 Note that the class $a_i$ with $\del a_i = \gamma$ had the
property that $\mu(a_i)=2$, and moreover $\alpha'(a_i)$ is a nonzero
element of geometric degree $n$ in $\widehat{\cH}$. So the homology in
degree $n$ of $\Lambda_\gamma(L)$ must be nonzero. But combining
Lemma~\ref{lem:top1} and Lemma~\ref{lem:top2}, we see that
$\Lambda_\gamma L$ is homotopy equivalent to the $n$-manifold $\tilde
L$. The nonvanishing of its top-dimensional homology now implies that
$\tilde L$ is closed, which in turns means that $\tilde L \to L$ is a
finite covering space.  
\end{proof}

The more precise statement in dimension 3 can be proven with some
specific results from 3-dimensional topology. I wish to thank K.~Fukaya, K.~Honda and S.~Maillot for helpful correspondence, which lead to the following proof of Corollary~\ref{cor:fukaya1}.

\begin{proof}(of Corollary~\ref{cor:fukaya1})
Let $L$ be a compact, orientable, prime 3-manifold. It is well-known
(see e.g. \cite[chapter 3]{He}) that either $L\cong S^1 \times S^2$ or
$L$ is irreducible, meaning that every embedded two-sphere in $L$
bounds a ball in $L$. 

If an irreducible 3-manifold $L$ admits a Lagrangian embedding into 
$\C^3$, then by Gromov's Theorem~\ref{thm:gromov-notexact} it has
infinite first homology, and hence infinite fundamental group, and so its
universal cover $\tilde L$ is non-compact. Moreover, by the sphere
theorem (see \cite[chapter 4]{He}), an irreducible 3-manifold has
trivial second homotopy group. It follows that $H_k(\tilde L)=0$ for
$k \geq 1$, and so by Hurewicz's theorem $\pi_k(\tilde L)=0$ for $k
\geq 1$, implying that $L$ itself is aspherical. Now by
Theorem~\ref{thm:fukayamain}, a finite cover of $L$ is homotopy
equivalent to $S^1 \times \Sigma$ for some closed oriented surface
$\Sigma$, and a result of Waldhausen~\cite[Corollary 6.5]{Wa:64}
implies that this homotopy equivalence can be improved to a
homeomorphism. 

Recall from the proof above that the fundamental group
$C$ of the cover arises as the centralizer of an element $\gamma \in
\pi_1(L)$ with $\mu(\gamma)=2$. 
Now I will argue that in fact $\gamma$ is central in $\pi_1(L)$, so
that the covering projection is actually a homeomorphism. Indeed, consider 
the exact sequence 
$$
0 \to K \to \pi_1(L) \stackrel{\frac 1 2 \mu}{\longrightarrow} \Z \to
0,
$$
where $K= \ker \mu$. Since the centralizer $C$ of $\gamma$ is of
finite index, $K'=C \cap K$ is of finite index in $K$. From the above
proof of the theorem, we see that $K'$ is finitely generated (it is
the fundamental group of $\Sigma$), so $K$ is also
finitely generated. Then by Stallings' fibration theorem (\cite{St:62},
see also \cite[Theorem~11.6]{He}), we deduce that $K$ is the fundamental
group of a compact surface $S$, which under our current assumptions
must be closed of genus at least 1. 

Now the proof concludes with the following observation:

\begin{lemma}
Any automorphism $\varphi$ of the fundamental group $K$ of a closed oriented
surface which is trivial on some finite index subgroup $K'$ is trivial.
\end{lemma}

\begin{proof}
If the surface is a sphere there is nothing to prove, 
so we consider the case that the genus of the surface is at least 1.
One knows that the fundamental group $K$ of a closed surface has no torsion.
Let $g\in K$ be given and consider the infinite cyclic subgroup $Z =
\langle g \rangle$ generated by $g$. The subgroup $\varphi(Z) \cap Z$
contains the finite index subgroup $K' \cap Z$, on which $\varphi$
acts trivially. 
But any automorphism of an infinite cyclic group fixing some nontrivial 
subgroup must be the identity, so $\varphi(g)=g$. Since this applies to any $g\in K$, the lemma is proven.
\end{proof}
Applying the lemma to the action of $\gamma$ on $K$ by conjugation,
which clearly fixes all elements of  $K' = K \cap C$, we
finally conclude that $\gamma$ is central in $\pi_1(L)$, so that $C=\pi_1(L)$, 
which finishes the proof of the theorem.
\end{proof}

\section{Reflections}

Above, I have presented Fukaya's elegant arguments
leading to some substantial new results about Lagrangian submanifolds
in $\C^n$. At first glance, it seems that string topology is really
essential to the approach. However, on further inspection, one
discovers that there may be a way to avoid it almost entirely.

The basic idea is the following. By Viterbo's theorem (see
\cite{Vi:98} and chapter chap:Viterbo), the homology of the free loop space  can be described in symplectic terms as the symplectic homology of the cotangent bundle. This theory can be defined in more general situations, for example for exact symplectic manifolds with contact-type
boundary. Moreover, for \textit{exact} codimension 0 embeddings $U
\hookrightarrow W$ one has restriction maps $SH_*(W) \to SH_*(U)$. 
It seems reasonable to expect (and is the subject of current work)
that every algebraic structure that exists on the homology of the free
loop space can also be defined on symplectic homology in general, even
if the underlying domain is not a cotangent bundle. In the exact case,
the restriction homomorphism should respect all these structures.

But even more should be true. In the case of a non-exact embedding $U
\hookrightarrow W$, there will be a Maurer-Cartan element in
$SH_*(U)$ such that after twisting all the structures by this
Maurer-Cartan element we get a morphism from $SH_*(W)$ to the twisted
version $SH_*^{\operatorname{twisted}}(U)$. This expectation is
consistent with (and gives one of several possible conceptual
explanations for) the results of Fukaya for Lagrangians in $\C^n$.

Indeed, with $U$ being a small neighborhood of the zero section in the
cotangent bundle of $L$ and $W=\C^n$, we are exactly in the situation
just described. What I have argued in earlier sections is that, after
twisting by a Maurer-Cartan element coming from the embedding, the
unit $[L] \in H_*(\Lambda L)$ with respect to the loop product has
become exact, which by a standard argument will force the twisted
homology to vanish completely. This is good news, because only in this
case can we even expect to have a morphism from $SH_*(\C^n)=0$ to this
ring. The prediction is that this morphism can indeed be defined in a
suitable chain version of the theory.

Once the above argument has been made to work, it extends the
applicability of Fukaya's approach in several directions. Notice that
string topology only enters indirectly, via Viterbo's isomorphism. 
As long as the algebraic operations can be defined and the morphism
associated to a codimension zero embedding respects them, one does not
even need to know that the operations on symplectic homology are the same as
those in string topology (although this is of course expected to be true).
Moreover, one can study non-exact codimension 0 embeddings
of general exact symplectic manifolds with contact boundary by this method.

\section{Guide to the literature}

The basic source for this chapter are of course Fukaya's
papers \cite{Fu:06,Fu:07}. Versions of Theorem~\ref{thm:fukayamain} under additional assumptions, like monotonicity of the Lagrangian submanifold, are much easier to achieve, see e.g. \cite{Bu:10, Da:12, EK:11} and the references therein.

For an introduction to symplectic topology the book \cite{MS:98} is recommended. It covers a lot more than is necessary to understand the problem  discussed here, and it gives some hints why Lagrangian submanifolds are so central in symplectic topology.
To learn something more specific about Lagrangian embeddings and
immersions, the excellent survey \cite{ALP:94} is still the best place
to start. Recently, new results have appeared which suggest that the problem in higher dimensions is more flexible than previously expected~\cite{EEMS:13}.

A chain complex $\mathcal C$ for the free loop space on which the loop bracket is fully defined, and which therefore might serve in the implementation of Theorem~\ref{thm:blackbox}, has recently been proposed by Irie~\cite{Ir:14}.

Finally, the reader who really wants to appreciate the discussion in
this chapter needs to know quite a bit about holomorphic curves. One
good source which thoroughly covers a lot of the basics, including a
version of Gromov compactness and a complete proof of Gromov's Theorem
\ref{thm:gromov-notexact}, is \cite{MS:04}. Many aspects of the theory are also covered in the earlier book~\cite{AL:94}. With these as
a guide, the monumental \cite{FOOO} will hopefully look less daunting.

\end{document}